\documentclass{article}
\usepackage[utf8]{inputenc}
\usepackage{amsthm}
\usepackage{amsmath}
\usepackage{amsfonts}
\usepackage{latexsym}
\usepackage{geometry}
\usepackage{enumerate}
\usepackage{csquotes}
\usepackage{graphicx}
\usepackage{comment}
\usepackage{appendix}
\usepackage{hyperref}
\usepackage{bm}

\emergencystretch=\maxdimen
\def \W {\mathcal{W}}
\def \N {\mathcal{N}}

\def \bW {\overline{\mathcal{W}}}
\def \dist {{\rm dist}}
\DeclareMathOperator{\Rm}{Rm}
\DeclareMathOperator{\Ric}{Ric}
\DeclareMathOperator{\Vol}{Vol}
\DeclareMathOperator{\spt}{spt}
\makeatletter
\newcommand*{\rom}[1]{\rm {\expandafter\@slowromancap\romannumeral #1@}}
\makeatother

\numberwithin{equation}{section}
\newtheorem{Theorem}{Theorem}[section]
\newtheorem{Proposition}[Theorem]{Proposition}
\newtheorem{Lemma}[Theorem]{Lemma}
\newtheorem{Corollary}[Theorem]{Corollary}
\theoremstyle{definition}

\title{A uniform Sobolev inequality for ancient Ricci flows with bounded Nash entropy}
\author{Pak-Yeung Chan, Zilu Ma, and Yongjia Zhang}

\begin{document}

\maketitle

\begin{abstract}
    This note is a continuation of \cite{CMZ21}. We shall show that an ancient Ricci flow with uniformly bounded Nash entropy must also have uniformly bounded $\nu$-functional. Consequently, on such an ancient solution there are uniform logarithmic Sobolev and Sobolev inequalities. We emphasize that the main theorem in this paper is true so long as the theory in \cite{Bam20c} is valid, and in particular, when the underlying manifold is closed. 
\end{abstract}

\section{Introduction}

Qi S Zhang \cite{Zhq07} proved a uniform Sobolev inequality for Ricci flows on closed manifolds. His method depends on Perelman's monotonicity formula. Indeed, Perelman's $\mu$-functional is known to be increasing in time, and by this reason, is bounded from below by the geometric data of the initial manifold. On the other hand, the $\mu$-functional is the logarithmic Sobolev constant. The Sobolev inequality then follows from these facts. 

If the underlining Ricci flow is an ancient solution, then there is no such ``initial manifold'' to make use of when estimating the $\mu$-functional. Nevertheless, Perelman's asymptotic soliton \cite[Proposition 11.2]{Per02} provides a substitute. This is exactly how \cite[Theorem1.4]{CMZ21} is proved. In that theorem, the authors show that, on an ancient solution admitting an asymptotic soliton, the $\mu$-functional at any time and at any scale is bounded from below by the entropy of the asymptotic shrinker. This also means that the $\nu$-functional is uniformly bounded.

However, Perelman's asymptotic soliton is not known to exist in general cases. This fact is a great restriction to the application of the logarithmic Sobolev and Sobolev inequalities proved in \cite[Theorem 1.4, Corollary 1.5]{CMZ21}. In this article, we shall remove this restriction by showing that Bamler's tangent flow at infinity \cite[Theorem 1.41]{Bam20c}, which is known to be a metric soliton, can also be used to estimate the $\mu$-functional for the ancient solution. For most of the definitions which appear in this article, the reader will find a brief introduction in \cite[Section 2]{CMZ21}; the notations in \cite{CMZ21} are completely adopted in this article.

Let $(M^n,g(t))_{t\in(-\infty,0]}$ be a complete ancient Ricci flow. Throughout this article, we shall make a technical assumption that $g(t)$ has bounded curvature within each compact time interval, namely,
\begin{eqnarray}\label{curvaturebound}
\sup_{M\times[t_1,t_2]}\big|{\Rm}_{g_t}\big|<\infty \quad \text{ for all }\quad -\infty<t_1\leq t_2\leq 0.
\end{eqnarray}
Note that the curvature bound may depend on the interval $[t_1,t_2]$, and is not uniform on $M\times(-\infty,0]$. We further assume the existences of a point $p_0$ $\in M$ and a  constant $Y\in(0,\infty)$, such that for all $\tau>0$, the Nash entropy based at $(p_0, 0)$ is bounded from below by $-Y$, i.e.,
\begin{equation}\label{nashuniformbound}
    \mathcal{N}_{p_0,0}(\tau)\geq -Y \quad \text{  for all }\quad \tau>0. 
\end{equation}

To apply the results in \cite{Bam20c}, we make an additional assumption: if the sequence in \cite[(1.2)]{Bam20c} is taken to be $$\left\{\big((M,g_i(t))_{t\in(-\infty,0]},(\nu^i_t)_{t\in(-\infty,0]}\big)\right\}_{i=1}^\infty,$$ where $\tau_i\to \infty$, $g_i(t):=\tau_i^{-1}g(\tau_it)$, and $\nu^i_t=\nu_{p_0, 0 \,|\,\tau_i t}$, then the theorems in \cite{Bam20c} (especially \cite[Theorem 1.6]{Bam20c} and all the statements contained in \cite[Theorem 1.15]{Bam20c}) are still valid. In particular, this is the case if $M^n$ is closed. Furthermore, we strongly believe that all the results in \cite{Bam20c} can be verified with only minor modifications under the assumptions (\ref{curvaturebound}) and (\ref{nashuniformbound}) alone. In other words, the assumption that the manifolds in \cite[(1.2)]{Bam20c} are closed can be replaced by (\ref{curvaturebound}).

The main theorem of this article is:
\begin{Theorem}\label{nu-functional}
Under the assumptions stated above, we have
\begin{eqnarray}\label{nu_nonsense_00}
\inf_{t\leq 0}\nu(g(t))=\mu_\infty>-\infty,
\end{eqnarray}
where $$\mu_\infty:=\inf_{\tau>0}\mathcal{N}_{p_0,0}(\tau)=\lim_{\tau\rightarrow\infty }\mathcal{N}_{p_0,0}(\tau)\geq -Y.$$
In particular, this is true for closed ancient Ricci flows with bounded Nash entropy.
\end{Theorem}

Once the $\nu$-functional is known to be bounded, the following logarithmic Sobolev inequalities and Sobolev inequalities are simply the consequences of some straightforward computations (c.f. \cite{Zhq07, LW20}).

\begin{Corollary}[The logarithmic Sobolev and Sobolev inequalities]
Under the assumptions as stated above, we have
\begin{enumerate}[(1)]
    \item Logarithmic Sobolev inequality: for any compactly supported locally Lipschitz function $u$ on $(M,g(t))$, where $t\leq 0$, and positive scale $\tau>0$, we have
    \begin{align*}
    \int_M u^2\log u^2dg_t-\left(\int_Mu^2dg_t\right)\int_M u^2dg_t+\left(\mu_\infty+n+\frac{n}{2}\log(4\pi\tau)\right)\int_Mu^2dg_t&
    \\
    \leq \tau\int_M(4|\nabla u|^2+Ru^2)dg_t&.
    \end{align*}
    \item Sobolev inequality: for any compactly supported locally Lipschitz function $u$ on $(M,g(t))$, where $t\leq 0$, we have
    \begin{equation*}
    \left(\int_M |u|^{\frac{2n}{n-2}}dg_t\right)^{\frac{n-2}{n}}\leq C(n)e^{-\frac{2\mu_\infty}{n}}\int_M(4|\nabla u|^2+Ru^2)dg_t.
    \end{equation*}
\end{enumerate}
Here $\mu_\infty$ is the lower bound of the $\nu$-functional in (\ref{nu_nonsense_00}). In particular, this is true for closed ancient Ricci flows with bounded Nash entropy.
\end{Corollary}

As an application of the Sobolev inequality, we shall also prove the following volume growth lower bound for steady gradient Ricci solitons. Note that we do not make any curvature assumptions in the following theorem.

\begin{Theorem}
\label{thm: steady vol growth}
Let $(M^n,g,f)$ be a complete steady gradient Ricci soliton normalized in the way that
$$\Ric = \nabla^2 f,\quad
    R+|\nabla f|^2=1.$$
Suppose that the induced ancient Ricci flow satisfies the same assumptions as in Theorem \ref{nu-functional}.
Fix a point $o\in M.$
Then 
\[
|B_r(o)|
\ge c(n) e^{\mu_\infty} r^{n/2}\quad\text{ for all }\quad r>10,
\]
where $c(n)$ is a dimensional constant and $\mu_\infty$ is as defined in Theorem \ref{nu-functional}. Here, we denote by $|\Omega|$ the volume of a measurable subset $\Omega\subset M$ and by $B_r(o)$ the geodesic ball centered at $o$ with radius $r.$
\end{Theorem}

This result generalizes the previous results on volume growth obtained in \cite{MS13} without too much extra restrictions. And the volume lower bound here is in a similar form as that for shrinking Ricci solitons proved in \cite[Proposition 6]{LW20}.

In dimension three, it is a consequence of the Hamilton-Ivey pinching estimate that any complete steady gradient Ricci soliton has nonnegative sectional curvature (see \cite{CBl09} and the references therein). When the curvature of the soliton is strictly positive, i.e., the manifold does not split locally, a result of Catino-Mastrolia-Monticelli \cite{CMM16} gives the following lower bound of the volume growth:
\begin{equation}\label{lower vol est 3d}
    |B_r(o)|\geq cr^2\quad\text{ for all $r$ large enough}.
\end{equation}
It will be interesting to see if the quadratic growth in (\ref{lower vol est 3d}) is sharp or not. Under the additional condition that the scalar curvature attains its maximum, we show that the steady gradient Ricci soliton must have quadratic volume growth. We do not impose any non-collapsed condition. In particular, the result also applies to the flying-wing examples constructed by Lai \cite{Lai20}.

\begin{Theorem}\label{steady 3d vol}
Suppose $(M^3, g, f)$ is a three dimensional complete steady gradient Ricci soliton with positive sectional curvature. Assume that the scalar curvature attains its maximum somewhere on $M$. Then there exists a positive constant $C$ such that for all large $r$, it holds that
\begin{equation}\label{quad vol}
C^{-1}r^2\leq |B_r(o)|\leq Cr^2.
\end{equation}
\end{Theorem}

\textbf{Acknowledgement.}
The first-named author was partially supported by an AMS–Simons Travel Grant and would like to thank Professor Jiaping Wang for fruitful discussions on the volume of steady solitons. The second-named author would like to thank Yuxing Deng for very enlightening discussions about steady solitons during his stay in La Jolla. 

\section{Proof of Theorem 1.1}
For any sequence $\tau_i\to \infty,$ there is a subsequence of
$g_i(t):=\tau_i^{-1}g(\tau_i t)$ which has an $\mathbb{F}$-limit in sense of \cite{Bam20b}, i.e.,
\begin{eqnarray}\label{F limit assumption}
\left((M,g_i(t))_{t\in(-\infty,0]},(\nu^i_t)_{t\in(-\infty,0]}\right)\xrightarrow{\makebox[1cm]{$\mathbb{F},\mathfrak{C}$}}\left(\mathcal{X}^{\infty},(\nu^{\infty}_t)_{t\in (-\infty,0)}\right) \quad\text{  as } \quad i\to \infty,
\end{eqnarray}
where $\nu^i_t=\nu_{p_0, 0 \,|\,\tau_i t}$ is the conjugate heat kernel, $\left(\mathcal{X}^{\infty},(\nu^{\infty}_t)_{t\in (-\infty,0)}\right)$ is a metric flow pair over $(-\infty,0]$, and $\mathfrak{C}$ is a correspondence (c.f. \cite[Definition 5.4]{Bam20b}). Note that this fact is true assuming only (\ref{curvaturebound}) and (\ref{nashuniformbound}), because it follows from \cite[Theorem 7.8]{Bam20b}, and does not depend on \cite{Bam20c}. According to our assumptions at the beginning of section 1, we shall then fix a sequence $\{\tau_i\}_{i=1}^\infty$, such that \cite{Bam20c} is valid for the sequence in (\ref{thefirstconvergence}).

\begin{Theorem}(\cite{Bam20c}) \label{thefirstconvergence}
Under the assumptions indicated at the begining of section 1, $\left(\mathcal{X}^{\infty},(\nu^{\infty}_t)_{t\in (-\infty,0)}\right)$ is a metric soliton and $\mathcal{X}^\infty_0=\{x^\infty\}$ consists of a singular point. Furthermore, we have $d\nu^\infty_t=d\nu_{x^\infty\,|\,t}$ and 
\begin{eqnarray}\label{convergenceofthenashentropy}
\lim_{i\rightarrow\infty}\mathcal{N}_{p_0,0}^{g_i}(\tau)=\mathcal N_{x^\infty}(\tau):=\int_{\mathcal{R}^\infty}f^\infty(\cdot,-\tau)d\nu^{\infty}_{-\tau}-\frac{n}{2}\quad\text{ for all }\quad \tau>0,
\end{eqnarray}
where $\mathcal N_{p_0,0}^{g_i}(\tau)$ is the Nash entrpy of the Ricci flow $g_i$, $\mathcal{R}^\infty$ is the regular part of $\mathcal{X}^\infty$, $f^\infty$ is a smooth function on $\mathcal{R}^\infty$ defined as $d\nu^\infty_t:=(4\pi|t|)^{-\frac{n}{2}}e^{-f^\infty(\cdot,t)}dg_{\infty,t}$ on $\mathcal{R}_{\infty}$, and $\nu^\infty_t(\mathcal{X}^\infty\setminus\mathcal{R}^\infty)=0$ for all $t<0$. The convergence in (\ref{F limit assumption}) is smooth on $\mathcal{R}^\infty$ in the sense of \cite[Theorem 9.31]{Bam20b}.
\end{Theorem}
\begin{proof}
By \cite[Theorem 1.4]{Bam20c}, the singular part $\mathcal{X}^\infty\setminus\mathcal{R}^\infty$ always has zero measure. By \cite[Theorem 1.6]{Bam20c}, the convergence in (\ref{F limit assumption}) is smooth on $\mathcal{R}^\infty$. The convergence of the Nash entropy (\ref{convergenceofthenashentropy}) follows from \cite[Theorem 1.15]{Bam20c}. Since the Nash entropy is monotonically decreasing in $\tau$, we have that
$$\mathcal N_{x^\infty}(\tau)=\lim_{i\rightarrow\infty}\mathcal{N}_{p_0,0}^{g_i}(\tau)=\lim_{i\rightarrow\infty}\mathcal{N}^g_{p_0,0}(\tau_i\tau)$$
is a constant independent of $\tau$. It then follows from \cite[Theorem 1.19]{Bam20c} that $\left(\mathcal{X}^{\infty},(\nu^{\infty}_t)_{t\in (-\infty,0)}\right)$ is a metric soliton.
\end{proof}

Next, to show that the $\nu$-functional is uniformly bounded on $(M,g(t))_{t\in(-\infty,0]}$, we shall use the same method as we have applied in \cite[Section 9]{CMZ21}. Let us arbitrarily fix $t_0\in(-\infty,0]$, $\tau_0>0$, and $u_0$ satisfying $u_0\geq 0$, $\sqrt{u_0}\in C_0^\infty(M)$, and $\int_M u_0 dg_{t_0}=1$. We shall estimate $$\overline{\W}(g(t_0),u_0,\tau_0):=\int_M\left(\tau_0\left(\frac{|\nabla u_0|^2}{u_0^2}+R_{g_{t_0}}\right)-\log u_0-\frac{n}{2}\log(4\pi\tau_0)-n\right)u_0\,dg_{t_0}.$$ As in \cite[Section 9]{CMZ21}, we shall, without loss of generality, assume $t_0<0$. 

Let us solve the conjugate heat equation coupled with $(M,g(t))_{t\in(-\infty,t_0]}$, with $u_0$ being its initial value. Then the solution is
\begin{eqnarray*}
u(x,t):=\int_M K(\cdot,t_0\,|\,x,t)u_0\,dg_{t_0}\quad\text{ for all }\quad (x,t)\in M\times(-\infty,t_0].
\end{eqnarray*}
By the fact that the conjugate heat equation preserves the integral, we also have that $$d\mu_t:=u(\cdot,t)\,dg_t\quad\text{ for all }\quad t\leq t_0$$ is a conjugate heat flow. Our first observation is the following.

\begin{Lemma}\label{theW1distancebound}
We have
\begin{eqnarray}
\sup_{t\in(-\infty,\tau_i^{-1}t_0]}\dist_{W_1}^{g^i_{t}}(\mu^i_t,\nu^i_t)\leq\frac{C}{\sqrt{\tau_i}},
\end{eqnarray}
where $\mu^i_t:=\mu_{\tau_i t}$, $\nu^i_t=\nu_{p_0, 0 \,|\,\tau_i t}$, $\dist_{W_1}$ is the $1$-Wassernstein distance, and $C$ is a constant depending on $u_0$ but independent of $i$.
\end{Lemma}
\begin{proof}

Let $\nu_t=\nu_{p_0, 0\,|\,t}$. By \cite[Proposition 3.24(b)]{Bam20b}, for all $t\leq t_0$, we have
\begin{equation}\label{W1 dist bdd1}
    \dist_{W_1}^{g_t}(\mu_t, \nu_t)\leq  \dist_{W_1}^{g_{t_0}}(\mu_{t_0}, \nu_{t_0}).
\end{equation}
Hence, we need only to estimate the right-hand-side of (\ref{W1 dist bdd1}). Let us fix a large positive constant $D_0$ such that $\spt u_0\subseteq B_{t_0}(p_0, D_0)$. Let $\phi$ be an arbitrary bounded $1$-Lipschitz function with respect to the metric $\dist_{g_{t_0}}$, then we may compute
\begin{eqnarray*}
\int_M \phi\, d\mu_{t_0}  - \int_M \phi\, d\nu_{t_0}&=& \int_M (\phi-\phi(p_0))\, d\mu_{t_0} - \int_M (\phi-\phi(p_0))\, d\nu_{t_0}\\
&\leq& D_0 \int_M u_{t_0}dg_{t_0}- \int_M (\phi-\phi(p_0))K(p_0, 0|\cdot, t_0)\, dg_{t_0}\\
&\leq& D_0 +\int_M \dist_{t_0}(\cdot, p_0)K(p_0, 0|\cdot, t_0)\, dg_{t_0}\\
&\leq& C,
\end{eqnarray*}
where we have used the Gaussian upper bound in \cite[Theorem 26.25]{RFV3} and (\ref{curvaturebound}) in the last inequality. Here $C$ is a constant depending only on the curvature bound on $M\times[t_0,0]$ and the lower bound of $\Vol_{g_0}\big(B_{g_0}(p_0,1)\big)$. It follows from the Kantorovich-Rubinstein Theorem that 
\begin{equation}\label{W1 dist bdd2}
\begin{split}
    \dist_{W_1}^{g_{t_0}}(\mu_{t_0}, \nu_{t_0})
    &= \sup_{\phi}\left( \int_M \phi\, d\mu_{t_0} 
    - \int_M \phi\, d\nu_{t_0}\right)\\
    &\leq C,
    \end{split}
\end{equation}
where the supremum is taken over all bounded $1$-Lipschitz functions $\phi$ with respect to the metric $\dist_{g_{t_0}}$. 

Finally, let us arbitrarily fix $i\in\mathbb{N}$ and $t\in(-\infty,\tau_i^{-1}t_0]$. Recall that the $W_1$-Wassernstein distance between two probability measures $\mu, \nu\in \mathcal{P}(X)$, where $X$ is a metric space, is defined as $$\dist_{W_1}(\mu,\nu)=\inf_{q}\int_{X\times X} \dist(x,y)\, dq(x,y),$$
where the infimum is taken over all couplings $q\in\mathcal{P}(X\times X)$ of $\mu$ and $\nu$. Hence, using (\ref{W1 dist bdd1}) and (\ref{W1 dist bdd2}), we can pick a coupling $q$ of $\mu^i_t(=\mu_{\tau_it})$ and $\nu^i_t(=\nu_{\tau_it})$ such that 
\begin{equation*}
\int_{M\times M}\dist_{g_{\tau_i t}}(x, y)\,dq(x,y)\leq 2C.
\end{equation*}
Therefore, we have
\begin{eqnarray*}
\dist_{W_1}^{g_{i, t}}(\mu^i_t, \nu^i_t)&\leq& \int_{M\times M}\dist_{g_{i, t}}(x, y)\,dq(x,y)\\
&=&\frac{1}{\sqrt{\tau_i}}\int_{M\times M}\dist_{g_{\tau_i t}}(x, y)\,dq(x,y)\\
&\leq&\frac{2C}{\sqrt{\tau_i}}.
\end{eqnarray*}
\end{proof}

The next observation follows from the exactly same reasoning as \cite[Proposition 9.5]{CMZ21}

\begin{Lemma}\label{upperbound}
There is a constant $C_0$, depending on $p_0$ and the function $u_0$, such that
\begin{eqnarray*}
u(x,t)\leq C_0K(p_0,0\,|\,x,t) \quad\text{ for all }\quad (x,t)\in M\times(-\infty,t_0).
\end{eqnarray*}
\end{Lemma}

For the convenience of the proof, we shall next fix some notations. Let
\begin{eqnarray*}
d\mu_t&:=&(4\pi|t|)^{-\frac{n}{2}}e^{-\bar f(\cdot,t)}dg_t\quad\text{ for }\quad t< t_0,
\\
d\mu^i_t&:=&(4\pi|t|)^{-\frac{n}{2}}e^{-\bar f^i(\cdot,t)}dg_t\quad\text{ for }\quad t< \tau_i^{-1}t_0,
\\
d\nu_t&:=&(4\pi|t|)^{-\frac{n}{2}}e^{ -f_{p_0,0}(\cdot,t)}dg_t\quad\text{ for }\quad t< 0,
\\
d\nu^i_t&:=&(4\pi|t|)^{-\frac{n}{2}}e^{ -f^i_{p_0,0}(\cdot,t)}dg_t\quad\text{ for }\quad t< 0,
\end{eqnarray*}
and we define
\begin{eqnarray*}
\N_{p_0,0}^i(\tau)&:=&\N_{p_0,0}^{g_i}(\tau)=\int_M f^i_{p_0,0}(\cdot,-\tau)d\nu^i_{-\tau}-\frac{n}{2}\quad\text{ for }\quad \tau>0,
\\
\overline{\N}^i(\tau)&:=&\int_M \bar f^i_{p_0,0}(\cdot,-\tau)d\mu^i_{-\tau}-\frac{n}{2}\quad\text{ for }\quad \tau>\tau_i^{-1}|t_0|.
\end{eqnarray*}

\begin{Lemma}
We have
\begin{eqnarray}\label{thesecondFconvergence}
\left((M,g_i(t))_{t\in(-\infty,\tau_i^{-1}t_0]},(\mu^i_t)_{t\in(-\infty,\tau_i^{-1}t_0]}\right)\xrightarrow{\makebox[1cm]{$\mathbb{F},\mathfrak{C}$}}\left(\mathcal{X}^{\infty},(\nu^{\infty}_t)_{t\in (-\infty,0)}\right)\quad \text{  as } \quad i\to \infty,
\end{eqnarray}
where $\mathfrak{C}$ is the same correspondence as in (\ref{F limit assumption}). The convergence is smooth on $\mathcal{R}^\infty$ in the sense of \cite[Theorem 9.31]{Bam20b}, and in particular, we have
\begin{eqnarray}\label{convergenceoff}
f^i_{p_0,0}\rightarrow f^\infty,\quad \bar f^i\rightarrow f^\infty
\end{eqnarray} locally smoothly on $\mathcal{R}^\infty$. Here both $f^\infty$ and $\mathcal{R}^\infty$ are defined in the statement of Theorem \ref{thefirstconvergence}, and, in the convergence, $f^i_{p_0,0}$ and $\bar f^i$ should be understood as being pulled back by the diffeomorphisms provided by \cite[Theorem 9.31]{Bam20b}.
\end{Lemma}
\begin{proof}

Combining Lemma 2.2 and \cite[Lemma 5.19]{Bam20b}, we obtain (\ref{thesecondFconvergence}). By \cite[Lemma 6.17]{Bam20b}, we then have
\begin{eqnarray}\label{CHFconvergence}
(\mu^i_t)_{t\in(-\infty,\tau_i^{-1}t_0]}\xrightarrow{\makebox[1cm]{$\mathfrak{C}$}}(\nu^{\infty}_t)_{t\in (-\infty,0)}\quad \text{  as } \quad i\to \infty,
\end{eqnarray}
where the convergence is in the sense of \cite[Definition 6.14]{Bam20b}. By \cite[Theorem 1.6]{Bam20c}, we have that the convergence in (\ref{thefirstconvergence}) is smooth on $\mathcal{R}^\infty$. Combining this fact with (\ref{CHFconvergence}) and applying \cite[Theorem 9.31(f)]{Bam20c}, we have that the convergence of $\mu_t^i$ is also locally smooth on $\mathcal{R}^\infty$. Hence, the convergence in (\ref{thesecondFconvergence}) is smooth on $\mathcal{R}^\infty$. This finishes the proof of the lemma.

\end{proof}

\begin{Lemma}
We have
\begin{eqnarray}
\lim_{i\rightarrow\infty}\N^i_{p_0,0}(\tau)=\lim_{i\rightarrow\infty} \overline{\N}^i(\tau)\equiv\N_{x^\infty}(\tau)\quad \text{ for all }\quad \tau>0.
\end{eqnarray}
All the terms above are constants independent of $\tau>0$.
\end{Lemma}
\begin{proof}
Since the convergences in (\ref{thefirstconvergence}), (\ref{thesecondFconvergence}), and (\ref{convergenceoff}) are all smooth on $\mathcal{R}^\infty$, and since $\nu^\infty_t(\mathcal{X}^\infty\setminus\mathcal{R}^\infty)\equiv 0$, one may argue in the same way as the proof of \cite[Theorem 14.45(a)]{Bam20c}. The only statement we need to check is: fixing any $\tau>0$, there is a constant $C$ independent of $i$, such that
\begin{eqnarray}\label{nonsense501}
\int_M e^{-\tfrac{1}{2}\bar f^i(\cdot,-\tau)}dg_{i,-\tau_i}\leq C\quad\text{ for all $i$ large enough}.
\end{eqnarray}

We shall now prove this estimate. By Lemma \ref{upperbound}, we have
$$e^{-\tfrac{1}{2}\bar f^i(\cdot,-\tau)}\leq C_0^{\frac{1}{2}}e^{-\tfrac{1}{2} f^i_{p_0,0}(\cdot,-\tau)}.$$
Hence, \cite[Proposition 5.5]{Bam20c} implies that
\begin{align*}
    \int e^{-\tfrac{1}{2}\bar f^i(\cdot,-\tau)}dg^i_{-\tau}&\leq \int C_0^{\frac{1}{2}}e^{-\tfrac{1}{2} f^i_{p_0,0}(\cdot,-\tau)}dg^i_{-\tau}=C_0^{\frac{1}{2}}(4\pi\tau)^{\frac{n}{2}}\int e^{\tfrac{1}{2} f^i_{p_0,0}(\cdot,-\tau)}d\nu^i_{-\tau}
    \\
    &\leq C_0^{\frac{1}{2}}(4\pi\tau)^{\frac{n}{2}}e^{\frac{n}{2}},
\end{align*}
where we have also used the fact that $R\geq 0$ on an ancient solution (c.f. \cite{CBl09}); this is exactly (\ref{nonsense501}). The rest of the proof is the same as the proof of \cite[Theorem 14.45(a)]{Bam20c}.
\end{proof}

In order to estimate $\overline{\W}(g_{t_0},u_0,\tau_0)$, we shall now defined another set of notations. Let 
\begin{eqnarray}\label{someofthedefinition}
\tau_t&:=&t_0+\tau_0-t\quad\text{for all }\quad t\leq t_0
\\\nonumber
u(\cdot,t)&:=&(4\pi\tau_t)^{-\frac{n}{2}}e^{-f(\cdot,t)}\quad\text{for all }\quad t\leq t_0,
\\\nonumber
\overline{\W}(t)&:=&\overline{\W}\big(g_t,u(\cdot,t),\tau_t\big)\quad\text{for all }\quad t\leq t_0,
\\\nonumber
\N(t)&:=&\int_{M}f(\cdot,t)u(\cdot,t)dg_{t}-\frac{n}{2}\quad\text{for all }\quad t\leq t_0.
\end{eqnarray}

The following lemma is the same as \cite[Lemma 9.1, Theorem 9.3]{CMZ21}.

\begin{Lemma}
We have $$\frac{d}{dt} \overline{\W}(t)\geq 0 \quad\text{ for all }\quad t<t_0$$ and $$\lim_{t\rightarrow t_0-}\overline{\W}(t)=\overline{\W}(g_{t_0},u_0,\tau_0).$$
\end{Lemma}

\begin{Lemma}\label{thelimitofthenashentropy}
We have
$$\lim_{t\rightarrow \infty}\N(\tau_it)=\lim_{i\rightarrow \infty}\overline{\N}^i(\tau)\equiv \N_{x^\infty}(\tau).$$
All the terms above are constants independent of $t<0$ or $\tau>0$.
\end{Lemma}
\begin{proof}
Let us fix an arbitrary $t<0$. Since
\begin{eqnarray*}
 f(\cdot,\tau_it)&=&\bar f^i(\cdot,t)-\frac{n}{2}\log\left(\frac{\tau_{\tau_it}}{\tau_i|t|}\right)
 \\
 &=&\bar f^i(\cdot,t)-\frac{n}{2}\log\left(1+\frac{\tau_0+t_0}{\tau_i|t|}\right),
\end{eqnarray*}
we have, by the definition of $\N$ and $\overline{\N}^i$,
$$\N(\tau_it)=\overline{\N}^i(|t|)-\frac{n}{2}\log\left(1+\frac{\tau_0+t_0}{\tau_i|t|}\right)\quad\text{ for all $i$ large enough}.$$
Taking $i\rightarrow\infty$, the lemma then follows immediately.
\end{proof}

With the preparations above, we are ready to show that the limit in Lemma \ref{thelimitofthenashentropy} is a lower bound of $\overline{\W}(g_{t_0},u_0,\tau_0)$ and thereby prove our main theorem.

\begin{proof}[Proof of Theorem \ref{nu-functional}]
In the proof, we shall retain the notations in (\ref{someofthedefinition}). First of all, we compute
\begin{eqnarray*}
-\frac{d}{dt}\mathcal{N}(t)&=& -\frac{d}{dt} \int_M fu\,dg_t=-\int_M u\Box f dg_t+\int_M  f \Box^*u \,dg_t\\
&=&-\int_M u\Box f \,dg_t=\int_M \left(-\frac{\partial f}{\partial t}+\Delta f\right) u\,dg_t\\
&=&\int_M \left(2\Delta f-|\nabla f|^2+R-\frac{n}{2\tau_t}\right)u\,dg_t\\
&=&\int_M \left(|\nabla f|^2+R\right)u\,dg_t-\frac{n}{2\tau_t}.
\end{eqnarray*}
Hence 
\begin{eqnarray*}
    -\tau_t\frac{d}{d t}\mathcal{N}(t)&=&\int_M \left(\tau_t\left(|\nabla f|^2+R\right)+f-n\right)u\,dg_t-\mathcal{N}(t)\\
    &=& \bW(t)-\mathcal{N}(t),
\end{eqnarray*}
and $$-\frac{d}{d t}\left(\tau_t\mathcal{N}(t)\right)=\bW(t).$$ By the monotonicity of $\bW(t)$, we have
\begin{align*}
&\frac{1}{\tau_i}\big(\tau_{-2\tau_i}\mathcal{N}(-2\tau_i)-\tau_{-\tau_i}\mathcal{N}(-\tau_i)\big)
\\
=&-\frac{1}{\tau_i}\int_{-2\tau_i}^{-\tau_i}\frac{d}{d t}\big(\tau_t\mathcal{N}(t)\big) dt =\frac{1}{\tau_i}\int_{-2\tau_i}^{-\tau_i}\bW(t) dt \leq \bW(g_{t_0},u_0,\tau_0).
\end{align*}
Letting $i\to \infty$, and using (\ref{thelimitofthenashentropy}), we have
\begin{equation}
\N_{x^\infty}(\tau)=\lim_{i\rightarrow\infty}\left(\frac{\tau_{-2\tau_i}}{\tau_i}\mathcal{N}(-2\tau_i)-\frac{\tau_{-\tau_i}}{\tau_i}\mathcal{N}(-\tau_i)\right)\leq \bW(g_{t_0},u_0,\tau_0).
\end{equation}
Since $t_0$, $u_0$, and $\tau_0$ are arbitrarily fixed, this finishes the proof of the theorem.
\end{proof}

\section{Applications to Steady Ricci Solitons}

Let $(M^n,g,f)$ be a complete steady gradient Ricci soliton satisfying
\[
    \Ric = \nabla^2 f,\quad
    R+|\nabla f|^2=1.
\]
By \cite{CBl09}, $R\ge 0$ everywhere on $M$, and as a consequence $$|\nabla f|\leq 1.$$ A complete steady gradient Ricci soliton generates a canonical solution $g(t)$ to the Ricci flow. If $\phi_t$ is the flow of the vector field $-\nabla f$ with $\phi_0=\text{  id }$, then $\phi_t$ exists for all time because of $|\nabla f|\leq 1$, and $(g(t):=\phi_t^*g)_{ t\in \mathbb{R}}$ solves the Ricci flow equation with $g(0)=g$. We prove a volume lower bound for steady gradient solitons whose canonical form $g(t)$ satisfy all the assumptions in Theorem \ref{nu-functional}. The argument only requires a  Sobolev inequality on $M$, so we shall show a slightly more general statement:

\begin{Proposition}
\label{prop: sob vol}
Suppose that we have a  Sobolev inequality on $(M^n,g,f):$
\[
\left(\int_M 
|u|^{\frac{2n}{n-2}}\, dg
\right)^{\frac{n-2}{n}}
\le C_{\rm Sob}
\int_M 4|\nabla u|^2 + Ru^2\, dg,
\]
for any compactly supported and locally Lipschitz function $u$, where $C_{\rm Sob}<\infty$ is the Sobolev constant.
Fix a point $o\in M.$
Then there is a constant $c$ depending on $n$ and $C_{\rm Sob}$ such that
\[
    |B_r(o)|\ge c r^{n/2}\quad
    \text{ for all }\quad r>10. 
\]
\end{Proposition}
\begin{proof}
The proof follows verbatim as Theorem 3.1.5 in \cite{SC} with minor changes. For simplicity, we write
\[
    V(r) := |B_r(o)|.
\]
We choose a Lipschitz test function
\[
    u(x)
    := (r - d(x,o))_+.
\]
Then
\begin{align*}
 \left(\int_M 
|u|^{\frac{2n}{n-2}}\, dg
\right)^{\frac{n-2}{n}}
&\ge \frac{r^2}{4} V(r/2)^{\frac{n-2}{n}},\\
|\nabla u|^2 &\le 1,\ \text{a.e.},\\
\int_M Ru^2
= \int_{B_r(o)} u^2\Delta f 
&= \int_{\partial B_r(o)}  u^2 \partial_r f 
-\int_{B_r(o)} \nabla f \cdot \nabla u^2\\
&\le 2\int_{B_r(o)} u
\le 2rV(r),
\end{align*}
where we integrated by parts and used the fact that $|\nabla f|\le 1.$
Hence, 
\[
\frac{r^2}{4} V(r/2)^{a}
\le C_{\rm Sob} (4+2r)V(r)
\le 4C_{\rm Sob} (1+r)V(r),
\]
where $a=\frac{n-2}{n}.$ So
\begin{align*}
   V(r)\ge c_1\frac{r}{2}  V(r/2)^{a},&\quad \text{ for } r\ge 1;\\
V(r)\ge 
c_1\frac{r^2}{4} V(r/2)^{a},&\quad
\text{ for } r\in (0,1),
\end{align*}
where $c_1=\min\left\{\frac{1}{16C_{\rm Sob}},1\right\}.$
Fixing $r>10$ and let $k=k(r)\ge 1$ be the integer such that
\[
    1\le 2^{-k}r < 2.
\]
Iterating the inequality above, we have
\begin{align*}
    V(r)
    \ge (c_1r)^{\sum_{j=0}^k a^j}
    2^{-\sum_{j=0}^k (j+1)a^j }
    V\left( 2^{-1-k}r\right)^{a^{1+k}}.
\end{align*}
Write $s=2^{-1-k}r\in [1/2,1).$
For any $m\ge 1$, we have
\begin{align*}
    V(s)
    &\ge
    c_1 (s/2)^2
    V(s/2)^a\\
    \ge \cdots&\ge (c_1s^2)^{\sum_{i=0}^m a^i}
    2^{-\sum_{i=0}^m 2(i+1)a^i}
    V(2^{-1-m}s)^{a^{1+m}}.
\end{align*}
Since $a\in (0,1),$ we have
\[
\lim_{m\to \infty}
 V\left( 2^{-m}s\right)^{a^{m}}= 1.
\]
Therefore, taking $m\to \infty,$ we have
\[
V(s)
\ge c(n)c_1^{n/2} s^n
\ge c(n) c_1^{n/2},
\]
where we have used the fact $s\in [1/2,1).$
So
\[
V(r)
\ge 
c(n) c_1^{n/2} r^{\frac{n}{2}(1-a^{1+k})}.
\]
By the definition of $k$, we have
\[ r^{-\frac{n}{2}a^{1+k(r)}} \ge c(n)\quad\text{ if }\quad r>10.
\]
Hence, for $r>10,$ we have
\[
    V(r) \ge c(n)c_1^{n/2} r^{n/2}.
\]

\end{proof}

\textbf{Remark.} The estimate on the integral involving $R$ are similar to Theorem 5.1 in \cite{MS13} (c.f.
Lemma 4.3 in \cite{D16}).

\begin{proof}[Proof of Theorem \ref{thm: steady vol growth}]
From the proof of Proposition \ref{prop: sob vol}, we can see that if the Sobolev constant is $C_{\rm Sob}=C(n)e^{-\frac{2\mu_\infty}{n}}$, then the constant $c$ in the previous proposition can be taken as 
\[
    c = c(n) C_{\rm Sob}^{-n/2}
    = c(n) e^{\mu_\infty}.
\]
This finishes the proof of Theorem \ref{thm: steady vol growth}.
\end{proof}

We end this section by establishing the quadratic volume estimate of positively curved three dimensional steady gradient Ricci soliton.



\def \n {\mathbf{n}}

\begin{proof}[Proof of Theorem \ref{steady 3d vol}]
Catino-Mastrolia-Monticelli \cite[Corollary 1.7]{CMM16} showed that a 3 dimensional steady gradient soliton with $\liminf_{r\to\infty}r^{-2}|B_r(o)|=0$ must either be flat or split isometrically as a quotient of $\mathbb{R}\times \Sigma$, where $\Sigma$ is the cigar soliton. The lower estimate in (\ref{quad vol}) then follows from the positively curved condition. To get the upper bound on the volume, we look at the area growth of the level sets. Since $\nabla ^2 f=\Ric>0$ and $R$ attains its maximum, by a result of Cao-Chen \cite[Proposition 2.3]{CC12}, there exist positive constants $\alpha$ $\in (0,1)$ and $c_0$ such that for any $x$ $\in M$, it holds that
\begin{equation}\label{f proper}
    \alpha\, \dist(x, o)-c_0\leq f(x)\leq \dist(x, o)+c_0.
\end{equation}
By the convexity of $f$ and the identity $\nabla R=-2\Ric(\nabla f)$, $f$ and $R$ have the same unique critical point, say $p_0$, then $R(p_0)=1-|\nabla f|^2(p_0)=1$ and $\max_M R=1$. Moreover, by the Morse Lemma, the level sets $\Gamma_s:=\{x:\, f(x)=s\}$ are all diffeomorphic to $\mathbb{S}^2$ for all $s>\min_M f$. Using $\Ric>0$ and $\nabla R=-2\Ric(\nabla f)$, we can find $\varepsilon_0>0$ and a large positive constant $s_0$ such that on $\{x:\, f(x)\geq s_0\}$, it holds that
$$R\leq \max_M R-\varepsilon_0=1-\varepsilon_0\,;$$
\begin{equation}\label{naf bdd}
    |\nabla f|^2=1-R\geq \varepsilon_0.
\end{equation}
As $\Gamma_s$ is a level set of $f$, its second fundamental form $A_s$ (w.r.t. the normal $\frac{\nabla f}{|\nabla f|}$) is given by $\frac{\nabla ^2 f}{|\nabla f|}=\frac{\Ric}{|\nabla f|}\geq 0$. For any $q$ $\in \Gamma_s$, we can find an orthonormal frame $\{e_i\}_{i=1}^2$ which is an eigenbasis of $A_s$ with eigenvalues $\sigma_i\geq 0$, $i=1, 2$. Let $\tilde{h}_s$ be the induced metric on $\Gamma_s$. Then by the Gauss equation
\begin{equation}\label{Gauss eqn}
\begin{split}
  K_s&=R_{1221}+ \sigma_1\sigma_2> 0\,;\\
  2K_s&=R-2\Ric(\n,\n)+2\sigma_1\sigma_2\\
      &\geq R-2\Ric(\n,\n),
\end{split}
\end{equation}
where $K_s$ and $\n$ denote the Gauss curvature of $(\Gamma_s, \tilde{h}_s)$ and the normal vector $\frac{\nabla f}{|\nabla f|}$, respectively. We then consider the flow $\psi_s$ of the vector field $\frac{\nabla f}{|\nabla f|^2}$ with $\psi_{s_0}=\text{  id }$. When restricted on $\Gamma_{s_0}$, $\psi_s: \Gamma_{s_0}\longrightarrow \Gamma_{s}$ are diffeomorphisms for all $s\geq s_0$. Let $h_s$ be the pull back metric $\psi_s^*\tilde{h}_s$ on $\Gamma_{s_0}$. Then we may compute
\begin{eqnarray*}
\frac{\partial}{\partial s}h_s&=&\psi_s^*\mathcal{L
}_{\frac{\nabla f}{|\nabla f|^2}}g=\psi_s^*\left(\frac{2\nabla ^2 f}{|\nabla f|^2}\right)=\psi_s^*\left(\frac{2\Ric}{|\nabla f|^2}\right),
\end{eqnarray*}
where $\mathcal{L}_{\frac{\nabla f}{|\nabla f|^2}}$ is the Lie derivative with respect to $\frac{\nabla f}{|\nabla f|^2}$. 
We denote by $dh_s$ the volume form induced by the metric $h_s$, then, applying (\ref{naf bdd}) and (\ref{Gauss eqn}), we have
\begin{equation}\label{vol form}
\begin{split}
\frac{\partial }{\partial s}dh_s&= \psi_s^*\left(\frac{R-\Ric(\n,\n)}{|\nabla f|^2}\right)dh_s\\
&\leq \psi_s^*\left(\frac{2K_s+\Ric(\n,\n)}{|\nabla f|^2}\right)dh_s\\
&\leq \varepsilon_0^{-1}\psi_s^*(2K_s+\Ric(\n,\n))\,dh_s.
\end{split}
\end{equation}
By virtue of $\nabla R=-2\Ric(\nabla f)$, we have $\psi_s^*(\Ric(\n,\n))=-\frac{1}{2}\frac{\partial }{\partial s}R\circ \psi_s$. It follows from (\ref{vol form}) and $R\leq 1$ that 
\begin{equation}
\begin{split}
\frac{\partial }{\partial s}\left(e^{\frac{R\circ \psi_s}{2\varepsilon_0}}\,dh_s\right) &\leq \varepsilon_0^{-1}\psi_s^*(2e^{\frac{R}{2\varepsilon_0}}K_s)\,dh_s\\
                                                                          &\leq \varepsilon_0^{-1}e^{\frac{1}{2\varepsilon_0}}\psi_s^*(2K_s)\,dh_s.
\end{split}
\end{equation}
Hence by the Gauss-Bonnet Theorem, we have
\begin{equation}
\begin{split}
\frac{d}{ds}\int_{\Gamma_{s_0}} e^{\frac{R\circ \psi_s}{2\varepsilon_0}}\,dh_s &= \int_{\Gamma_{s_0}}\frac{\partial }{\partial s}\left(e^{\frac{R\circ \psi_s}{2\varepsilon_0}}\,dh_s\right)\\
&\leq \frac{8\pi e^{\frac{1}{2\varepsilon_0}}}{\varepsilon_0}.
\end{split}
\end{equation}
Integrating the above differential inequality with respect to $s$, we can choose a $s_1>s_0$ such that for all $s\geq s_1$, it holds that
\begin{eqnarray*}
\text{ Area }(\Gamma_s)&\leq& \int_{\Gamma_{s_0}} e^{\frac{R\circ \psi_s}{2\varepsilon_0}}\,dh_s\\
&\leq& \frac{8\pi e^{\frac{1}{2\varepsilon_0}}}{\varepsilon_0} (s-s_0)+e^{\frac{1}{2\varepsilon_0}}\text{ Area }(\Gamma_{s_0})\\
&\leq& Cs.
\end{eqnarray*}
We used $R\geq 0$ in the first inequality. The Coarea formula then implies that for all $s\gg s_1$, it holds that
\begin{eqnarray*}
|\{x:\, f(x)\leq s\}|&=&|\{x:\, f(x)< s_1\}|+|\{x:\, s_1\leq f(x)\leq s\}|\\
                     &\leq& |\{x:\, f(x)< s_1\}|+\frac{Cs^2}{2\sqrt{\varepsilon_0}}\\
                     &\leq& C's^2.
 \end{eqnarray*}    
 By (\ref{f proper}), for all large $r$, we have
 $$|B_r(o)|\leq |\{x:\, f(x)\leq r+c_0\}|\leq 4C'r^2.$$
\end{proof}

\bibliography{bibliography}{}
\bibliographystyle{amsalpha}

\newcommand{\etalchar}[1]{$^{#1}$}
\providecommand{\bysame}{\leavevmode\hbox to3em{\hrulefill}\thinspace}
\providecommand{\MR}{\relax\ifhmode\unskip\space\fi MR }
\providecommand{\MRhref}[2]{%
  \href{http://www.ams.org/mathscinet-getitem?mr=#1}{#2}
}
\providecommand{\href}[2]{#2}

\bigskip
\bigskip

\noindent Department of Mathematics, University of California, San Diego, CA, 92093
\\ E-mail address: \verb"pachan@ucsd.edu "
\\

\noindent Department of Mathematics, University of California, San Diego, CA, 92093
\\ E-mail address: \verb"zim022@ucsd.edu"
\\

\noindent School of Mathematics, University of Minnesota, Twin Cities, MN, 55414
\\ E-mail address: \verb"zhan7298@umn.edu"

\end{document}